\documentclass{m2an}

\usepackage[colorlinks,linkcolor=blue,anchorcolor=blue,citecolor=red]{hyperref}

\newtheorem{theorem}{Theorem}

\newtheorem{lemma}[theorem]{Lemma}

\newtheorem{definition}{Definition}
\newtheorem{proposition}{Proposition}

\usepackage{bm}
\usepackage{float}
\usepackage{multirow}
\usepackage{upgreek}
\usepackage{amsfonts}
\usepackage{amssymb}
\usepackage{graphicx}
\usepackage{mathrsfs}
\usepackage{tabularx}
\usepackage{array}
\usepackage{cite}
\usepackage{amsthm}
\usepackage{amsmath}

\begin{document}
\title{Extending Azumaya algebras associated to arithmetic 2-bridge links}
\author{Jiayu Wan}\address{Rice University}

\begin{abstract}
Let $\Gamma$ be a finitely generated group and consider the set of all characters of representations of $\Gamma$ into $SL_2(\mathbb{C})$. This set, denoted by $X(\Gamma)$,  admits an algebraic structure and is called the character variety of $\Gamma$. When $\Gamma$ is the fundamental group of a hyperbolic 3-manifold $M$, $X(\Gamma)$ turns out to be a powerful tool in the study of the geometry and topology of $M$. Chinburg-Reid-Stover \cite{Chin_Reid_Stover} have borrowed tools from algebraic and arithmetic geometry to understand algebraic and number-theoretic properties of the canonical curves of $X(\Gamma)$. 
In this paper, we will partly generalize their results to certain hyperbolic link complements, and prove that the associated canonical quaternion algebra will not extend to an Azumaya algebra over the canonical surfaces.   
\end{abstract}

\date{}

\maketitle

\section{Introduction}

Let $\Gamma$ be a finitely generated group and let $X(\Gamma)$ be the $SL_2(\mathbb{C})$ character variety of $\Gamma$ (see Section 2.1). When $\Gamma$ is the fundamental group of a hyperbolic 3-manifold $M$ (such as the exterior of a hyperbolic knot complement in $S^3$), seminal work of Thurston(\cite{Thurston}) and Culler-Shalen(\cite{Culler_Shalen_1}) established $X(\Gamma)$ as a powerful tool in the study of the geometry and topology of $M$. For instance, in \cite{Thurston}, Thurston established the connection between the number of cusps of $M$ and the complex dimension of its canonical components. In \cite{Culler_Shalen_1}, Culler-Shalen showed how the character variety can be used to detect embedded essential surfaces in hyperbolic knot complements. 

More recently, Chinburg-Reid-Stover(\cite{Chin_Reid_Stover}) have borrowed tools from algebraic and arithmetic geometry to understand algebraic and number-theoretic properties of certain components of $X(\Gamma)$ as well as distinguished points on these components. In particular, they build a canonically defined quaternion algebra $A_{k(C)}$ over the function field of a canonical component of $X(\Gamma)$ (See Section 2.1). Based on this canonical quaternion algebra, their study is two-fold. On the one hand, they specialize the quaternion algebra at Dehn surgery points of the component and prove results about invariants of Dehn surgeries on hyperbolic knots. On the other hand, they study the Alexander polynomial of the hyperbolic knot and discover a connection between the roots of this polynomial and the extension of the quaternion algebra over the whole canonical component. Based on their work, several other results have been developed. For instance, Rouse(\cite{Rouse}) has used the canonical quaternion algebra to study Dehn surgeries of a specific hyperbolic knot. Miller(\cite{Miller}) derived similar results regarding the extension of the canonical quaternion algebra for once punctured torus bundles. 

In this paper, we are going to generalize the results of \cite{Chin_Reid_Stover} in another direction: moving from hyperbolic knot complements to hyperbolic link complements. The results of \cite{Chin_Reid_Stover}, \cite{Rouse}, \cite{Miller} deal with 1-cusped hyperbolic 3-manifolds and hence the canonical components of $X(\Gamma)$ are complex algebraic curves by Thurston's theorem. In this paper, we will consider 2-cusped hyperbolic 3-manifolds, in which case the canonical components are complex surfaces. In particular, we are going to study the three 2-bridge links, $5_{1}^{2}$, $6_{2}^{2}$, and $6_{3}^{2}$ in Rolfsen's table(see \cite{Rolfsen}). Note these are the only three arithmetic 2-bridge links(see \cite{Gehring}). Our main result is the following. 

\begin{theorem}\label{main result}
    Let $L$ be one of the three links, $5_{1}^{2}$, $6_{2}^{2}$, or $6_{3}^{2}$, and let $M=S^3 -L$ be the complement. Let $S$ be a canonical component of $X(\pi_{1}(M))$. Then $A_{k(S)}$ does not extend to an Azumaya algebra over the surface $S$.
\end{theorem}

Our result assumes that the canonical surfaces are defined over $\mathbb{C}$, and hence are complex surfaces. In contrast, in \cite{Chin_Reid_Stover}, the canonical curves are assumed to be defined over some number field. The reason for this distinction is that for curves, all the points with codimension 1 are closed points and hence the residue fields at those points are finite extensions of the base field $K$. So if the curve is defined over $\mathbb{C}$, i.e. $K=\mathbb{C}$, then all the residue fields will be $\mathbb{C}$ since $\mathbb{C}$ is algebraically closed. Then, except for the generic point of the curve,  the specialization of $A_{k(C)}$ at any other point will be trivial since $Br(\mathbb{C})=0$. Then there is nothing interesting in this case, hence the need for the curve to be defined over some number field. However, for a complex surface, the points with codimension 1 will correspond to curves on that surface and the residue fields at those points will not be trivial. Therefore, we can conclude something non-trivial even when the base field is $K=\mathbb{C}$. 

This paper consists of four sections: this section serves as a short introduction. In Section 2, we introduce character variety and its relation with cusped hyperbolic 3-manifolds. In Section 3, we summarize the definitions and key facts about Azumaya algebras and Brauer groups. In Section 4, we prove Theorem \ref{main result}. 
\section{Character varieties and cusped hyperbolic 3-manifolds}

\subsection{Character Varieties}
Character varieties have proven to be a powerful tool in the study of hyperbolic 3-manifolds, bridging the algebraic properties of the fundamental groups and the topology of the 3-manifolds. For important works in this direction, we refer the reader to \cite{PT2015, MPV2011, BK2013, LR2003}. In this section, we present a short introduction to character varieties. Let $\Gamma$ be a finitely generated group, and let $R(\Gamma)$ be the set of all representations of $\Gamma$ in $SL_2(\mathbb{C})$, i.e. $R(\Gamma)=Hom(\Gamma, SL_2(\mathbb{C}))$. As discussed in \cite{Culler_Shalen_1}, $R(\Gamma)$ is naturally endowed with the structure of an affine algebraic closed subset of $\mathbb{C}^{4n}$, and we call $R(\Gamma)$ (with this algebraic structure) the representation variety of $\Gamma$. Two elements $\rho_1, \rho_2 \in R(\Gamma)$ are called equivalent if there exists $g \in GL_2(\mathbb{C})$ such that $\rho_2 = g\rho_{1}g^{-1}$. A representation $\rho \in R(\Gamma)$ is called irreducible if there is no nontrivial subspace of $\mathbb{C}^2$ invariant under the action of $\rho(\Gamma)$, and is called reducible if it is not irreducible.


For any $\rho \in R(\Gamma)$, the character of $\rho$ is the function $\chi_{\rho}: \Gamma \rightarrow \mathbb{C}$ defined by $\chi_{\rho}(g) = tr(\rho(g))$. It is well-known 
that if $\rho_1,\rho_2 \in R(\Gamma)$ such that $\rho_1$ is irreducible and $\chi_{\rho_1} = \chi_{\rho_2}$, then $\rho_1,\rho_2$ are equivalent. Conversely, traces are always invariant under conjugacy (hence equivalence). So characters 'almost' classify representations up to equivalence. 
Let $X(\Gamma)$ be the set of all characters of $\Gamma$ in $SL_2(\mathbb{C})$, then $X(\Gamma)$ is also endowed with an algebraic structure, from the arguments in \cite{Culler_Shalen_1}. We therefore have the following definition. 

\begin{definition}
    Let $\Gamma$ be a finitely generated group and $X(\Gamma)$ the set of characters of $\Gamma$ in $SL_2(\mathbb{C})$. Then $X(\Gamma)$, together with the endowed algebraic structure, is called the character variety of $\Gamma$ in $SL_2(\mathbb{C})$.
\end{definition}

There is a natural surjective map $t: R(\Gamma) \rightarrow X(\Gamma)$ that sends each $\rho \in R(\Gamma)$ to its character $\chi_{\rho}$. It can be shown that $t$ is a regular map, if $R(\Gamma)$ and $X(\Gamma)$ are equipped with the corresponding algebraic structure. We have the following important facts regarding this regular map $t$. 

\begin{proposition}\label{facts about t map}
    (\cite{Culler_Shalen_1}, 1.4.2. and 1.4.4.) The following are true:
    
    (1) ~The set of reducible representations in $R(\Gamma)$ has the form $t^{-1}(V)$ for some closed algebraic subset $V$ in $X(\Gamma)$. 
    
    (2) ~If $R_0$ is an irreducible closed subset of $R(\Gamma)$ which contains an irreducible representation, then $t(R_0) \subset X(\Gamma)$ is an affine variety (i.e. it's closed and irreducible).
\end{proposition}

\noindent In this paper, we are more interested in some special irreducible components of $X(\Gamma)$ rather than all of $X(\Gamma)$. In fact, any irreducible component of $X(\Gamma)$ that contains the character of an irreducible representation is the image of some subvariety of $R(\Gamma)$ under the map $t$. This fact will be frequently used later in this article, and we shall present a short proof of it here. 

\begin{lemma}\label{irreducible component image}
    Let $C$ be an irreducible component of $X(\Gamma)$ (hence $C$ is closed and irreducible) that contains the character of an irreducible representation. Then there exists some irreducible closed subset of $D \subset R(\Gamma)$ such that $t(D) = C$. 
\end{lemma}

\begin{proof}
Since $t$ is a regular map, $t$ is continuous. Hence $t^{-1}(C)$ is closed as $C$ is closed. Write $t^{-1}(C) = D_1 \cup...\cup D_n$, where $D_1,...,D_n$ are the irreducible components of $t^{-1}(C)$. Since $t$ is surjective, $C=t(t^{-1}(C))=t(D_1 \cup...\cup D_n)=t(D_1)\cup...\cup t(D_n) $. Since $C$ is closed, we have $C=\overline{t(D_1)}\cup...\cup \overline{t(D_n)}$. Since $C$ is irreducible, $C=\overline{t(D_i)}$ for some $i$. By Proposition \ref{facts about t map}(2), $t(D_i)$ is closed, and hence $C=t(D_i)$. If $D_i$ contains purely reducible representations, then $D_i \subset S$ where $S$ is the set of all reducible representations in $R(\Gamma)$. Then by Proposition \ref{facts about t map}(1), $S=t^{-1}(V)$ for some closed $V$ in $X(\Gamma)$, and hence $t(S)=V$ as $t$ is surjective. So $C=\overline{t(D_i)} \subset V$. This contradicts the fact that $V$ consists purely of characters of reducible representations whereas $C$ contains the character of an irreducible representation by assumption. So $D_i$ must contain an irreducible representation, and by Proposition \ref{facts about t map}(2), $t(D_i)$ must be a variety and hence closed. So $C=\overline{t(D_i)}=t(D_i)$. This completes the proof. 
\end{proof}

\bigskip

\subsection{Tautological Representation and Canonical Quaternion Algebra}

Now, suppose $R_0$ is a closed subset of $R(\Gamma)$, and we fix an element $\gamma \in \Gamma$. Then every representation $\rho \in R_0$ gives a matrix $\rho(\gamma) \in SL_2(\mathbb{C})$, that can be written in the following form

\begin{equation}\label{representation of rho_gamma}
  \rho(\gamma) = 
      \begin{pmatrix}
        a_{\gamma}(\rho) & b_{\gamma}(\rho)\\
        c_{\gamma}(\rho) & d_{\gamma}(\rho)
      \end{pmatrix}.
\end{equation}

\noindent It can be shown that $a_{\gamma},b_{\gamma},c_{\gamma},d_{\gamma}$ are in fact regular functions on $R_0$ for every $\gamma \in \Gamma$. Therefore, there is a map $\mathcal{P}: \Gamma \rightarrow M_2(\mathbb{C}[R_0])$ given by

\begin{equation}\label{representation of P_gamma}
  \mathcal{P}(\gamma) = 
  \begin{pmatrix}
    a_{\gamma} & b_{\gamma}\\
    c_{\gamma} & d_{\gamma}
  \end{pmatrix}
\end{equation}

\noindent where $\mathbb{C}[R_0]$ is the ring of regular functions on $R_0$. Since every $\rho \in R_0$ is a representation, i.e. a group homomorphism, it can be easily shown that $\mathcal{P}$ is also a group homomorphism. Since  $a_{\gamma}(\rho)d_{\gamma}(\rho)-b_{\gamma}(\rho)c_{\gamma}(\rho)=det(\rho(\gamma))=1$ for any $\gamma \in \Gamma$ and $\rho \in R_0$, we have $\mathcal{P}(\gamma) \in SL_2(\mathbb{C}[R_0])$. In particular, if $R_0$ is irreducible so that we can define the function field $k(R_0)$ of $R_0$ (which is the fraction field of $\mathbb{C}[R_0]$), then we can view $\mathcal{P}$ as a representation in $SL_2(k(R_0))$. Following \cite{Chin_Reid_Stover}, we call this representation $\mathcal{P}: \Gamma \rightarrow SL_2(k(R_0))$ the \textit{tautological representation}. 

For an arbitrary field $F$ of characteristic 0, a representation $\Gamma \rightarrow SL_2(F)$ is called absolutely irreducible if it remains irreducible over an algebraic closure $\overline{F}$. It can be shown that the tautological representation is absolutely irreducible if $R_0$ contains some irreducible representation, using the lemma below. We refer the reader to \cite{Chin_Reid_Stover} for more detail. 

\begin{lemma}\label{characterization of reducible representation}
    (\cite{Culler_Shalen_1}, 1.2.2.) If $F$ is an algebraically closed field of characteristic 0. Then a representation $\rho: \Gamma \rightarrow SL_2(F)$ is reducible if and only if $\chi_{\rho}(c)=2$ for each element $c$ of the commutator subgroup of $\Gamma$.
\end{lemma}


Now, let $C$ be an irreducible component of $X(\Gamma)$ that contains the character of an irreducible representation. By Proposition \ref{facts about t map}, there exists an irreducible closed subset $D \subset R(\Gamma)$ such that $t(D)=C$. Clearly, $D$ must contain an irreducible representation. Let $k(D),k(C)$ be the function fields of $D,C$ respectively. Then we can view $k(D)$ as a field extension of $k(C)$, induced by the regular map $t$. Let $\mathcal{P}: \Gamma \rightarrow SL_2(k(D))$ be the tautological representation associated with $D$. Since $D$ is induced by $C$, we denote this representation by $\mathcal{P}_C$. This representation is absolutely irreducible by our previous arguments. We can then consider the $k(C)$-subalgebra of $M_2(k(D))$, denoted by $A_{k(C)}$, which is generated by the elements of $\mathcal{P}_C(\Gamma)$, i.e.

\begin{equation}\label{canonical quaternion algebra}
    A_{k(C)} =\{ \sum_{i=1}^{n}  \alpha_i \mathcal{P}_C(\gamma_i): \alpha_i \in k(C), \gamma_i \in \Gamma \}
\end{equation}

\noindent In fact, $A_{k(C)}$ is a quaternion algebra over $k(C)$ and assumes a special form of Hilbert symbol, which is summarized in the following theorem. 

\begin{theorem}\label{Hilbert symbol for A_KC}
    (\cite{Chin_Reid_Stover}, Corollary 2.9.) Let $\Gamma$ be a finitely generated group and $C$ an irreducible component of $X(\Gamma)$ that contains the character of an irreducible representation. Let $g,h \in \Gamma$ be two elements such that there exists a representation $\rho \in R(\Gamma)$ with character $\chi_{\rho} \in C$ for which the restriction of $\rho$ to $\langle g,h \rangle$ is irreducible. Let $A_{k(C)}$ be the $k(C)$-algebra defined above. Then $A_{k(C)}$ is a quaternion algebra over $k(C)$ with a Hilbert symbol: ($\frac{I_{g}^{2}-4,~ I_{[g,h]}-2}{k(C)}$).
\end{theorem}

\noindent This quaternion algebra $A_{k(C)}$ is called the canonical quaternion algebra of $C$, and it will be the central object of study in this article. 

\bigskip

\subsection{Cusped Hyperbolic 3-Manifolds}
We will be particularly interested in the case where $\Gamma=\pi_{1}(M)$ for a $M$ a hyperbolic 3-manifold. In this thesis, a hyperbolic 3-manifold will mean a connected, oriented and complete manifold $M$ of the form $\mathbb{H}^3/\Gamma$, where $\mathbb{H}^3$ is the hyperbolic 3-space and $\Gamma \cong \pi_1(M)$ is a torsion free discrete subgroup of $Isom^{+}(\mathbb{H}^3) \cong PSL_2(\mathbb{C})$. We refer the reader to \cite{Ratcliffe} and \cite{Thurston_2} for more detail on hyperbolic 3-manifolds. 

We will focus on the situation where $M$ is a knot or link exterior in $S^3$. In this case, M is noncompact, of finite volume, and has a finite union of incompressible tori as its boundary. By Mostow Rigidity Theorem(see \cite{Ratcliffe}, section 11.8), there is a unique (up to equivalence) discrete and faithful representation $\rho_0: \Gamma \rightarrow PSL_2(\mathbb{C})$, coming from the holonomy of the complete structure on $M$. This $\rho_0$ lifts to a representation $\hat{\rho_0}: \Gamma \rightarrow SL_2(\mathbb{C})$ according to the study of Thurston(\cite{Culler_Shalen_1}, 3.1.1.). In general, this lift won't be unique even up to equivalence. 

We define a canonical component $X_0(\Gamma) \subset X(\Gamma)$ to be an irreducible component of $X(\Gamma)$ containing the character of some $\hat{\rho_0}$ describe above, i.e. the character of a lift of the discrete faithful representation $\rho_0: \Gamma \rightarrow PSL_2(\mathbb{C})$. Notice that there may be multiple canonical components of $X(\Gamma)$ as the lift won't be unique. Any such canonical component is clearly a complex affine variety. Its dimension is given by the number of cusps of the manifold (in our case, it's the number of components of the link) as proved by Thurston in the following important theorem. 

\begin{theorem}\label{Thurston's theorem}
    (\cite{Culler_Shalen_1}, 3.2.1.) Let $M$ be a noncompact hyperbolic 3-manifold of finite volume with $d$ cusps and set $\Gamma=\pi_1(M)$. Then any canonical component $X_0(\Gamma)$ is a $d$-dimensional complex affine algebraic variety. In particular, when $M$ is a one (resp. two)-cusped hyperbolic 3-manifold, the canonical component $X_0(\Gamma)$ is a complex affine curve (resp. surface). 
\end{theorem}
\section{Azumaya Algebra and Brauer group}

\subsection{Azumaya Algebra and Brauer group over Fields}
In this section, we give a short introduction to Azumaya algebras over fields and schemes. We start with the more basic one: Azumaya algebras over fields. Let $k$ be a field, there are many equivalent ways to define an Azumaya algebra over $k$, we pick the one that is easier to memorize as the definition. A $k$-algebra (possibly noncommutative) $A$ is said to be central if the center of $A$ is $k$. It is said to be simple if the only 2-sided ideals of $A$ are $0$ and $A$ itself. 

\begin{definition}
    An Azumaya algebra over $k$ is a finite-dimensional, central and simple algebra over $k$.
\end{definition}

Let $A$ be an Azumaya algebra over $k$, then $A \cong M_n(D)$ for some division $k$-algebra, and $D$ is uniquely determined up to isomorphism(we refer the reader to chapter 2 of \cite{Gille} for more detail). We put an equivalence relation on the set of Azumaya algebras over $k$ as follows: two Azumaya algebras $A$, $B$ are equivalent, denoted by $A \sim B$, if $A \cong M_n(D)$ and $B \cong M_m(D)$ for some $m,n >0$. This equivalence relation is called Morita equivalence among Azumaya algebras over $k$. With Morita equivalence, we can give the definition of the Brauer group as follows. 

\begin{definition} 
Fix a field $k$. The Brauer group of $k$, denoted by $Br(k)$, is defined to be the equivalence classes of Azumaya algebras over $k$ under Morita equivalence, with group operation induced by the tensor product over $k$, ie: $[A] \cdot [B] = [A \otimes_{k}B]$, where $\cdot$ is the group operation on $Br(k)$ and $[A]$ stands for the equivalence class of $A$.
\end{definition}

\bigskip

\subsection{Azumaya Algebra and Brauer group over Schemes}
We have defined Azumaya algebras over an arbitrary field $k$. We now generalize this definition to Azumaya algebras over a commutative local ring. Let $R$ be a commutative local ring with residue field $k$, an algebra $A$ over $R$ is an Azumaya algebra if $A$ is free of finite rank $r \geq 1$ as an R-module and $A \otimes_{R} k$ is an Azumaya algebra over $k$. 

Next, we want to generalize this concept even further to define an Azumaya algebra over a scheme. Let $X$ be a Noetherian scheme. We assume familiarity with the structure sheaf (denoted by $\mathcal{O}_X$), stalk at a point $x$ (denoted by $\mathcal{O}_{X,x}$), residual field at a point $x$ (denoted by $k(x)$, coherent sheaf of modules, locally free sheaf of modules etc. An Azumaya algebra $\mathcal{A}$ on $X$ is a locally free sheaf of $\mathcal{O}_X$-algebras such that for every $x \in X$, $\mathcal{A}_{x}$ is an Azumaya algebra over the local ring $\mathcal{O}_{X,x}$. We will be interested in quaternion Azumaya algebras, which are Azumaya algebras that have rank 4 as locally free $\mathcal{O}_X$-modules. 

Similar to the equivalence of Azumaya algebras over fields, we can define the notion of equivalence between two Azumaya algebras over the scheme $X$. Let $\mathcal{A}$, $\mathcal{B}$ be two Azumaya algebras over $X$. $\mathcal{A}$, $\mathcal{B}$ are said to be equivalent if there exist locally free sheaves of $\mathcal{O}_X$-modules $\mathcal{F}_1$ and $\mathcal{F}_2$ such that $\mathcal{A} \otimes_{\mathcal{O}_X} End_{\mathcal{O}_X}(\mathcal{F}_1) \cong \mathcal{B} \otimes_{\mathcal{O}_X} End_{\mathcal{O}_X}(\mathcal{F}_2)$, where $End_{\mathcal{O}_X}(\mathcal{F})$ is the sheaf of $\mathcal{O}_X$-module endomorphisms of an $\mathcal{O}_X$-module $\mathcal{F}$. It can be shown that this is an equivalence relation. Similar to the Brauer group of a field $k$, we define the Brauer group of $X$, denoted by $Br(X)$, to be the group of equivalence classes of Azumaya algebras over $X$.

\bigskip
\subsection{Extension on Azumaya Algebras}
Given an Azumaya algebra $A$ over the function field $k(X)$, there are two types of extension of $A$ that are important to us: extension of $A$ over $X$ and extension of $A$ over codimension one points of $X$. We first consider extension of $A$ over $X$. If $f: X \rightarrow Y$ is a morphism of schemes, then there is an induced homomorphism $\tilde{f}:  Br(Y) \rightarrow Br(X)$ (see \cite{Poonen}, section 6.6 for more detail). Let $k(X)$ be the function field of the scheme $X$, ie: the residue field at the generic point of $X$. We know that there is a natural morphism $f: spec(k(X)) \rightarrow X$ sending the point $spec(k(X))$ to the generic point of $X$. Then this map induces:

\begin{equation*}
    \tilde{f}: Br(X) \rightarrow Br(spec(k(X))) \cong Br(k(X))
\end{equation*}

\noindent If $X$ is a regular integral Noetherian scheme, then $\tilde{f}$ is injective. An Azumaya algebra $A$ over $k(X)$ is said to extend to an Azumaya algebra $\mathcal{A}$ over $X$ if $[A] = \tilde{f}([\mathcal{A}])$ where $[A]$ stands for the equivalence class of $A$. Intuitively, an Azumaya algebra $\mathcal{A}$ over $X$ gives rise to an Azumaya algebra over the residue field at every point $x$ of $X$(which we call the specialization of $\mathcal{A}$ at $x$). If $A$ is an Azumaya algebra over $k(X)$, which is the residue field at the generic point of $X$, then $A$ extends to $\mathcal{A}$ over $X$ if the specialization of $\mathcal{A}$ at the generic point is isomorphic to $A$. 

Now, let $x$ be a codimension one point of $X$, then $R=O_{X,x}$ is a local ring with fraction field $k(X)$. An Azumaya algebra $A$ over $k(X)$ is said to extend over $x$ if there is an Azumaya algebra $A_R$ over $R$ such that $A$ is isomorphic to $A_R \otimes_{R} k(X)$. We have the following results relating extension of $A$ over codimension one points of $X$ and its extension over $X$, whose proofs can be found in \cite{Milne1980} and \cite{COP2002}. 

\begin{theorem}\label{extension of Azumaya algebra}
     Let $X$ be a regular integral Noetherian scheme of dimension at most 2 with function field $K$ that is quasi-projective over a field or a Dedekind ring, then:
    
    (1) ~An Azumya algebra $A_K$ over $K$ extends to an Azumaya algebra $\mathcal{A}$ over $X$ if and only if it extends over all codimension one points $x$ of $X$.
    
    (2) ~Let $x$ be a codimension one point of $X$ and let $A_K$ be a quaternion algebra over $K$ with Hilbert symbol $(\frac{\alpha,\beta}{K})$. The tame symbol of $A_K$ at $x$, denoted by $\{\alpha,\beta\}_{x}$, is defined to be the class of $(-1)^{ord_{x}(\alpha) ord_{x}(\beta)} \frac{\beta^{ord_{x}(\alpha)}}{\alpha^{ord_{x}(\beta)}}$ in $k(x)^{*}/(k(x)^{*})^2$, where $k(x)$ is the residue field at $x$. If $k(x)$ has characteristic different from 2, then this symbol is trivial if and only if $A_K$ extends over $x$.
\end{theorem}

If we let $\Gamma =\pi_1(M)$ be the fundamental group of some hyperbolic knot (or link) complement and $C$ be a canonical component of $X(\Gamma)$, then we can construct the canonical quaternion algebra $A_{k(C)}$ as described in section 2.2. A fundamental question is: 

\bigskip
\begin{center}
Does this Azumaya algebra over the function field of $C$ extend to an Azumaya algebra over $C$? 
\end{center}
\bigskip

\noindent \cite{Chin_Reid_Stover}  made progress in the case where $M$ is a hyperbolic knot complement and hence $C$ is an affine curve (by Theorem \ref{Thurston's theorem}). In this paper, we will partly generalize their results to hyperbolic link complements with two components, and hence $C$ will be affine surfaces. To be more specific, we will study the case where $M=S^3 - L$ for $L$ an arithmetic two-bridge link with two components. In fact, there are three such links: the Whitehead link $5_{1}^{2}=(8/3), ~6_{2}^{2}=(10/3)$, and $6_{3}^{2}=(12/5)$. We will show that for these 3 links, the canonical quaternion algebra $A_{k(S)}$ does not extend to the surface $S$. This will be the content of the next section.
\section{Proof of Theorem \ref{main result}}

\subsection{General Idea} 

We will prove our main result, which is Theorem \ref{main result}. In this subsection, we present the general approach. We then analyze the three arithmetic links one by one in the following subsections. 

Let $L$ be a two bridge link in $S^3$ with two components. Then the fundamental group $\pi_1(S^3-L)$ can be presented in the following form: 

\begin{equation}\label{presentation of fund group}
\pi_1(S^3-L)=\langle a,b~|~awa^{-1}w^{-1}=1 \rangle
\end{equation}

\noindent where $a,b$ are two meridional generators. Such a link will admit a reduced diagram by considering the projection of the link onto $\mathbb{R}^2$. The reduced diagram will yield two numbers, $\alpha$ and $\beta$, that completely classify the isotopy class of the link. We call $\alpha$ the torsion, and $\beta$ the crossing number of the link. Then, the link will be denoted by $(\alpha / \beta)$, which is called the Schubert normal form of the link(we refer the reader to \cite{Zieschang}, chapter 12, for more information on Schubert normal form). If the Schubert normal form of $L$ is given by $(\alpha /\beta)$, then the word $w$ has the form: 

\begin{equation}\label{expression of w}
w=b^{\epsilon_1}a^{\epsilon_2}b^{\epsilon_3} \cdots a^{\epsilon_{\alpha-2}}b^{\epsilon_{\alpha-1}},
\end{equation}

\noindent where $\epsilon_{i}=(-1)^{\lfloor \frac{i\beta}{\alpha} \rfloor}$, and $\lfloor x \rfloor$ stands for the floor function of $x$. We refer the reader to \cite{Rolfsen} and \cite{Zieschang} for a thorough treatment on the subject.  

Let $\Gamma=\pi_1(S^3-L)$, $X(\Gamma)$ the character variety of $\Gamma$, and let $C$ be a canonical component of $X(\Gamma)$. By Thurston's theorem(Theorem \ref{Thurston's theorem}), $C$ is a complex affine surface in this case since $L$ has two components. In order to avoid any confusion in terms of the notation, we replace $C$ by $S$ to denote the canonical surface. We want to get a clearer description of $S$, i.e. we want to compute $N$ such that $S \subset \mathbb{C}^N$ and the defining polynomials of $S$. To achieve this goal, we need the following lemmas regarding traces in $SL_2(\mathbb{C})$. 

\begin{lemma}\label{trace identities}
    (see \cite{Mac_and_Reid}, section 3.4.) Suppose $X,Y \in SL_2(\mathbb{C})$, then we have the following identities:

    \begin{equation}\label{trace identity equations}
        \begin{aligned}
            & tr(X) =tr(X^{-1}) \\
            & tr(XY) = tr(X)tr(Y)-tr(XY^{-1}) \\
            & tr([X,Y])=tr^2(X)+tr^2(Y)+tr^2(XY)-tr(X)tr(Y)tr(XY)-2
        \end{aligned}
    \end{equation}
\end{lemma}

\noindent From \eqref{trace identity equations}, it is easy to see that for $\Gamma$ as above, any $\chi_{\rho} \in X(\Gamma)$ is completely determined by $\chi_{\rho}(a),\chi_{\rho}(b)$ and $\chi_{\rho}(ab)$. Set $x:=\chi_{\rho}(a), ~y:=\chi_{\rho}(b),~ z:=\chi_{\rho}(ab)$. Then $X(\Gamma)$ will be an affine closed subset of $\mathbb{C}^3$, and hence $S$ will be an affine surface in $\mathbb{C}^3$ with coordinates $x,y,z$. 

We are interested in the cases where $L$ is one the three arithmetic 2-bridge links: $L=5_{1}^{2},~6_{2}^{2} ~or~ 6_{3}^{2}$. The Schubert normal forms of the three links are given by $(8/3)$, $(10/3)$, and $(12/5)$ respectively. If we let $A=a^{-1}$ and $B=b^{-1}$, then the word $w$ appeared in the presentation of the fundamental group for the three links is given by: 

\begin{equation}\label{w for links}
    \begin{aligned}
        & w_1=baBABab, \\
        & w_2=babABAbab, \\
        & w_3=baBAbabABab
    \end{aligned}
\end{equation}

\noindent which correspond to $5_{1}^{2}$, $6_{2}^{2}$, $6_{3}^{2}$ respectively. Thanks to the work of Harada(\cite{Harada}) and Landes(\cite{Landes}), there is a unique canonical surface in $X(\Gamma)$ when $L$ is one of these three links. These canonical surfaces will be denoted $S_1,S_2,S_3$ that corresponds to $5_{1}^{2}$, $6_{2}^{2}$, $6_{3}^{2}$ respectively, and are defined as the vanishing set of the polynomials $p_1,p_2,p_3$ where:

\begin{equation}\label{vanishing polynoimals for links}
    \begin{aligned}
        & p_1=z^3-xyz^2+(x^2+y^2-2)z-xy, \\
        & p_2=z^4-xyz^3+(x^2+y^2-3)z^2-xyz+1, \\
        & p_3=z^3-xyz^2+(x^2+y^2-1)z-xy
    \end{aligned}   
\end{equation}

Our main objective is to study the canonical quaternion algebra $A_{k(S)}$ for these three links. By Theorem \ref{Hilbert symbol for A_KC}, a Hilbert symbol for the canonical quaternion algebra is given by ($\frac{I_{g}^{2}-4,~ I_{[g,h]}-2}{k(S)}$). Since our $\Gamma$ is generated by two elements $a,b$, we can set $g=a$ and $h=b$. If we let the coordinates of $S$ be $x,y,z$ as described above, then from \ref{trace identity equations}, the Hilbert symbol is given by ($\frac{x^{2}-4,~ x^2+y^2+z^2-xyz-4}{k(S)}$). By Lemma \ref{characterization of reducible representation}, a representation $\rho: \Gamma \rightarrow SL_2(\mathbb{C})$ is reducible if and only $tr(\rho(c))=2$ for any $c \in [\Gamma,\Gamma]$. Hence, the vanishing set of $x^2+y^2+z^2-xyz-4$ will correspond to the set of characters of reducible representations(see also \cite{Petersen}, Theorem 1.1.). To show that the canonical quaternions doesn't extend, it suffices to show that the tame symbol along some codimension one point(in this case it's a curve) is nontrivial by Theorem \ref{extension of Azumaya algebra}. We will consider the algebraic set $C$ given by the intersection of $S$ with the set of characters of reducible representations, and show that the tame symbol along some irreducible curve in $C$ is nontrivial. For simplicity, we will always use $f$ to denote the defining polynomial of $S=S_1, S_2 ~or~ S_3$, and use $(\frac{\alpha,\beta}{k(S)})$ to denote the canonical quaternion algebra where 
\begin{equation}\label{def of alpha and beta}
    \begin{aligned}
        & \alpha=x^2-4 \\
        & \beta = x^2+y^2+z^2-xyz-4.
    \end{aligned}
\end{equation}
Hence, the algebraic set $C$ will be the vanishing set of both $f$ and $\beta$, i.e. 
\begin{equation}\label{def of C}
    C: f=\beta=0.
\end{equation}

Our proofs will be heavily based on facts in algebraic geometry. We assume that the reader is familiar with the important concepts in the subject, such as varieties (affine or projective, we always assume our varieties are irreducible), function fields, morphisms, rational maps, divisors and genus etc. For the remaining parts of the paper, all the varieties are assumed to be either affine or projective, ie: they lie in some affine space $A^n$ or some projective space $\mathbb{P}^n$. We use $K$ to denote the base field of a variety and $k(X)$ to denote the function field of a variety $X$. All the varieties will be defined over some algebraically closed field $K$ with characteristic 0. In fact, we will almost always be dealing with $K=\mathbb{C}$. We refer the reader to \cite{Hartshorne}, \cite{Shafarevich}, and \cite{Silverman} for more thorough treatment on algebraic geometry. 

\bigskip

\subsection{$5_{1}^{2}$ Case(Whitehead Link)}
In this case, the canonical surface $S_1$ is given by the equation $f=0$ where

\begin{equation}\label{equation for S_1}
f=z^3-xyz^2+(x^2+y^2-2)z-xy
\end{equation}

\noindent Substituting \eqref{def of alpha and beta} with $\beta=0$ into $f=0$ gives $z=\frac{xy}{2}$. Then, we substitute this equation back into $f$ and $\beta$ to see that the algebraic set $C$(vanishing set of $f$ and $\beta$) consists of four lines, given as follows:

 \begin{equation*}
    \begin{aligned}
        & L_1: x=2,y=z,\\
        & L_2: x=-2,y=-z, \\
        & L_3: y=2,x=z, \\
        & L_4: y=-2,x=-z. 
    \end{aligned}
\end{equation*}

\noindent We consider the line $L=L_3$ and compute the tame symbol along this line. Recall from Theorem \ref{extension of Azumaya algebra} that the tame symbol is given by $(-1)^{ord_{L}(\alpha) ord_{L}(\beta)} \frac{\beta^{ord_{L}(\alpha)}}{\alpha^{ord_{L}(\beta)}}$. Clearly $\alpha=x^2-4$ doesn't vanish entirely on $L$, so $ord_{L}(\alpha)=0$. The tame symbol then becomes $\frac{1}{\alpha^{ord_{L}(\beta)}}$, so it suffices to compute $ord_L(\beta)$(see \eqref{def of alpha and beta} for definition of $\beta$). To do this, we need the following lemma. 

\begin{lemma}\label{intersection multi and order}
    Let $X,Y \subset spec(\mathbb{C}[x,y,z])$ be surfaces given by defining polynomials $f,g$ respectively. Let $C \subset X \cap Y$ be an irreducible curve. The intersection multiplicity of $X,Y$ along $C$ is defined by $m_{C}(X,Y)=l_{\mathcal{O}_C}(\mathcal{O}_{C}/(f,g))$, where $\mathcal{O}_{C}$ is the local ring at the generic point of $C$, and $l_{R}(M)$ is the length of the $R$-module $M$(the longest length of a chain of submodules of $M$). Then $ord_{C}(f)=m_{C}(X,Y)$ when we view $f$ as a rational function on $Y$.
\end{lemma}

\begin{proof} Let $A=\mathbb{C}[x,y,z]$, then $X=spec(A/(f))$ and $Y=spec(A/(g))$. Let $\mathfrak{p}$ be the defining prime ideal of $C$, ie: $C = spec(A/\mathfrak{p})$. We write $f_{\mathfrak{p}},g_{\mathfrak{p}}$ to be the images of $f,g$ in $A_{\mathfrak{p}}$ respectively. Then by definition, when we view $f$ as a rational function on $Y$, $ord_C(f)=d$ where $(f_{\mathfrak{p}})=m^{d}$ for $m$ the maximal ideal in the local ring $(A/(g))_{\mathfrak{p}} \cong A_{\mathfrak{p}}/(g_{\mathfrak{p}})$. On the other hand, 

\begin{equation*}
    m_{C}(X,Y)=l_{\mathcal{O}_C}(\mathcal{O}_{C}/(f,g))=l_{A_{\mathfrak{p}}}(A_{\mathfrak{p}}/(f_{\mathfrak{p}},g_{\mathfrak{p}}))=l_{A_{\mathfrak{p}}/(g_{\mathfrak{p}})}((A_{\mathfrak{p}}/(g_{\mathfrak{p}}))/(f_{\mathfrak{p}}))=d
\end{equation*}

\noindent Hence $m_C(X,Y)=ord_{C}(f)$ when view $f$ as a rational function on $Y$. This completes the proof. 
\end{proof}

From Lemma \ref{intersection multi and order}, to compute $ord_{L}(\beta)$, it suffices to compute $m_C(S_1,Y)$ where $Y$ is the surface in $\mathbb{C}^3$ given by $\beta=0$. To compute the intersection multiplicity, we borrow the following lemma from Eisenbud-Harris(\cite{Eisenbud-Harris}). 

\begin{lemma}\label{inter_multi_trans}
    (\cite{Eisenbud-Harris}, Theorem 1.26.) Let $A,B \subset X$ be subvarieties of a smooth variety $X$ (affine or not) such that every irreducible component $C$ of $A \cap B$ has $codim(C)=codim(A)+codim(B)$, where $codim(A)$ stands for the codimension of a subvariety $A \subset X$. Then for each such $C$, the intersection multiplicity $m_C(A,B)=1$ if and only if $A,B$ intersect transversely at a general point of $C$, i.e. the points where $A,B$ intersect transversely form a dense subset of $C$.
\end{lemma}

By Lemma \ref{inter_multi_trans}, it suffices to check whether $S_1$ and $Y$ intersect transversely at a general point along $L$. Let's compute the normal vectors (the gradients) of $S_1$ and $Y$ along $L$. For $Y$, the normal vector is: 

\begin{equation*}
     \begin{pmatrix}
        2x-yz\\
        2y-xz\\
        2z-xy
      \end{pmatrix} =\begin{pmatrix}
        0\\
        4-x^2\\
        0
      \end{pmatrix}, ~when ~~y=2, ~x=z
\end{equation*}

\noindent For $S_1$, the normal vector along $L$ is given by: 

\begin{equation*}
     \begin{pmatrix}
        -2\\
        -x^3+3x\\
        2
      \end{pmatrix}
\end{equation*}

\noindent Clearly the two normal vectors are not collinear at any point with $x \neq \pm 2$ along $L$. Hence the two surfaces intersect transversely at a general point of $L$, and hence $ord_L(\beta)=m_L(S_1,Y)=1$.

So the tame symbol becomes $\frac{1}{\alpha}$. We want to decide if $\alpha$ iss a square in $k(L)$. But $\alpha$ is a regular function on $L$, and $L$ is a complex line and hence $L \cong \mathbb{C}$. In particular, the ring of regular functions on $L$ is a UFD and $\alpha=x^2-4=(x-2)(x+2)$ is the unique factorization of $\alpha$. Hence $\alpha$ can't be a square. So the tame symbol along this line $L\subset S_1$ is nontrivial, and hence the canonical quaternion algebra $A_{k(S_1)}$ doesn't extend to an Azumaya algebra over $S_1$.

\bigskip

\subsection{$6_{3}^{2}$ Case}
In this case, the surface $S_2$ is defined by the vanishing set of polynomial: 

\begin{equation}\label{equation for S_2}
f=z^3-xyz^2+(x^2+y^2-1)z-xy
\end{equation}

\noindent We first show that in this case the set $C$(given by \eqref{def of C}) is an elliptic curve, which is a curve of genus one(see \cite{Silverman} for a thorough discussion of elliptic curves). To achieve this goal, we borrow the following theorem from \cite{Shafarevich}. We refer the reader to \cite{Shafarevich} pp. 263-264 for a discussion of the tree of infinitely near points and the proof of the theorem.

\begin{theorem}\label{genus formula for singular curves}
    (\cite{Shafarevich}, pp. 264) Let $X \subset \mathbb{P}^2$ be a curve whose defining polynomial has degree $n$. Let $g$ be the genus of $X$ and $d_i$ be the multiplicity at a particular infinitely near point $i$, then $g = \frac{(n-1)(n-2)}{2} - \Sigma\frac{d_i(d_i-1)}{2}$, where the summation is taken over all the infinitely near points (of any singular point of $X$).
\end{theorem}

\begin{lemma}\label{ellipticity of C}
    $C$ is an elliptic curve.
\end{lemma}

\begin{proof}
We consider the projection map $P$ of $C$ onto the xy-plane, let $C^{'}=P(C)$ be the image of the projection. Notice that:

\begin{equation*}
f=z^3-xyz^2+(x^2+y^2-1)z-xy=\beta z+3z-xy.
\end{equation*}

\noindent Since $C$ is defined by $f=0,\beta=0$, we must have that $z=\frac{xy}{3}$ for any $(x,y,z)$ on $C$. Substituting $z=\frac{xy}{3}$ back into $\beta=0$, we get that $C^{'}$ can be defined by $h=0$ where $h=9x^2+9y^2-2x^2y^2-36$. It's clear that $P$ is an isomorphism with inverse $P^{-1}: (x,y) \rightarrow (x,y,\frac{xy}{3})$. So $g(C)=g(C^{'})$, where $g(C)$ is the genus of the curve $C$. It thus suffices to calculate the genus of $C^{'}$. 

Notice that $C^{'}$ is a smooth affine curve over the complex, whose projective completion in $\mathbb{P}^2$ is given by 
\begin{equation*}
    \tilde{C}: 9x^2t^2+9y^2t^2-2x^2y^2-36t^4=0.
\end{equation*}
\noindent It's easy to see that this completion $\tilde{C}$ has two singular points at infinity: $(1:0:0),(0:1:0)$. Since any affine curve is birational to its projective completion, we have $g(C^{'})=g(\tilde{C})$, and so it suffices to find $g(\tilde{C})$. 

Since $\tilde{C}$ is a plane curve in $\mathbb{P}^2$, the genus formula(Theorem \ref{genus formula for singular curves}) gives:

\begin{center}
$g(\tilde{C})=\frac{(n-1)(n-2)}{2}-\Sigma_{i}\frac{d_i(d_i-1)}{2}$,
\end{center}

\noindent where $n$ is the degree of the defining polynomial of $\tilde{C}$ and $d_i's$ are multiplicities of all infinitely near points. In our case, $n=4$, and we have two singular points at infinity. Notice that our defining polynomial of $\tilde{C}$ is symmetric in $x$ and $y$. So the two singular points at infinity will have the same tree of infinitely near points, and hence $\Sigma_{i}\frac{d_i(d_i-1)}{2}=2\Sigma_{j}\frac{d_j(d_j-1)}{2}$ where $d_j's$ are multiplicities corresponding to infinitely near points of $(1:0:0)$. In particular, $g(\tilde{C})=3-2t$ for some integer $t \geq 0$. Since genus is always nonnegative, we must have $g(\tilde{C})=1,t=1$ or $g(\tilde{C})=3,t=0$. But since $(1:0:0)$ is singular, there is at least one $d_j \geq 2$. So $t \geq 1$ and we must have $t=1, g(\tilde{C})=1$. So $g(C)=g(C^{'})=g(\tilde{C})=1$ and this completes the proof.
\end{proof}

Using the same arguments and similar computations as the previous example, we get that the inverse of the tame symbol along $C$ is again given by $\alpha=x^2-4$. We want to check if  $\alpha$ is a square in $k(C)^{\times}$. From the proof of Lemma \ref{ellipticity of C}, we know that $C$ is isomorphic to $C{'}$, its projection onto xy-Plane. Moreover, $\alpha$, when viewed as a rational function over $C^{'}$, has the same expression: $\alpha=x^2-4$. So it suffices to decide if $\alpha=g^2$ for some $g$ in $k(C^{'})$. We have $\alpha=(x+2)(x-2)$, and it's easy to compute 
\begin{equation*}
    \begin{aligned}
        & div(x+2)=2P_1-Q_1-Q_2, \\
        & div(x-2)=2P_2-Q_1-Q_2,
    \end{aligned}
\end{equation*}
\noindent where $div(h)$ represents the divisor class of the rational function $h$, $P_1=(-2,0),P_2=(2,0)$ and $Q_1,Q_2$ are two points at infinity. Hence $div(\alpha)=2P_1+2P_2-2(Q_1+Q_2)$. If $\alpha=g^2$, then $div(g)=P_1+P_2-Q_1-Q_2$. So it suffices to decide whether the divisor $D=P_1+P_2-Q_1-Q_2$ is principal or not. Since our curve $C'$ is an elliptic curve, we can apply the following criterion to check the principality of $D$.

\begin{theorem}\label{criteria for principal divisors}
    (\cite{Silverman}, Chapter III, Corollary 3.5.) Let $E$ be an elliptic curve and $D=\Sigma n_{P}(P) \in Div(E)$. Then $D$ is principal if and only if $\Sigma n_{P}=0$ and $\bigoplus [n_{P}]P = O$. (The first sum is of integers, and the second is addition on E).
\end{theorem}

Now, suppose $D$ is principal, then we must have $P_1 \oplus P_2 \ominus Q_1 \ominus Q_2=0$, which means $P_1 \oplus P_2=Q_1 \oplus Q_2$. But we know $2P_1-Q_1-Q_2=div(x+2)$ is principal and hence $P_1 \oplus P_1=Q_1 \oplus Q_2$. Thus $P_1 \oplus P_2=Q_1 \oplus Q_2=P_1 \oplus P_1$. Then we must have $P_1=P_2$, but clearly $P_1 \neq P_2$. So $D$ can't be principal, and $\alpha$ is not a square in $k(C')$. So $A_{k(S_2)}$ doesn't extend for the same reason as the example in the previous subsection.

\bigskip

\subsection{$6_{2}^{2}$ Case} In this case, the surface $S_3$ is defined to be the vanishing set of:

\begin{equation}\label{equation for S_3}
    f=z^4-xyz^3+(x^2+y^2-3)z^2-xyz+1
\end{equation}

\noindent We will show that in this case the smooth projective model of $C$(given by \eqref{def of C}) is a hyperelliptic curve of genus 3(notice that any two smooth projective models of $C$ will be isomorphic to each other, so the smooth projective model of $C$ is well-defined up to isomorphism). But before giving a proof of this claim, we first briefly introduce the concept of hyperelliptic curves. 

A smooth projective curve $X$ over some algebraically closed field $K$ is said to be hyperelliptic if $g(X) \geq 2$ and it admits a 2-1 morphism $\phi: X \rightarrow \mathbb{P}^{1}$. The Hurwitz formula(see \cite{Silverman}, Chapter II, Theorem 5.9) implies that there precisely $2g+2$ ramification points for $\phi$, where $g$ is the genus of the curve. These points are called the Weierstrass points in the setting of hyperelliptic curves. It's a well-known fact that every hyperelliptic curve admits an affine model of the form:

\begin{equation}\label{affine model for hyperelliptic curve}
    y^2 = f(x)  
\end{equation}

\noindent for some $f\in K[x]$ separable of degree $2g+1$ or $2g+2$. This form is called the Weierstrass normal form of the hyperelliptic curve. The curve can have either a smooth point or a singular point at infinity. If it has a singular point at infinity, we can resolve the singularity through blowups to get either one smooth point or two smooth points. So finally the curve can have either one or two smooth points at infinity. If it has one smooth point at infinity, it's called an imaginary hyperelliptic curve. Otherwise, it's called a real hyperelliptic curve. Every hyperelliptic curve admits an involution map that sends $(x,y)$ to $(x,-y)$, and the Weierstrass points are precisely the fixed points of this map. We denote the image of a point $P$ under the involution by $\bar{P}$. We now prove that in our case $C$ is hyperelliptic and compute its genus. 

\begin{lemma}\label{hyperellipticity of C}
    The smooth projective model of $C$ is a hyperelliptic curve of genus 3.
\end{lemma}

\begin{proof}
Notice that $f=\beta z^2+z^2-xyz+1$. Along $C$, $f=\beta=0$ by \eqref{def of C}, which leads to $z^2-xyz+1=0$. Substituting this equation into $\beta=0$, we get $x^2+y^2-5=0$. So $C$ can be defined by
\begin{equation*}
    \begin{aligned}
        & x^2+y^2-5=0 \\
        & z^2-xyz+1=0.
    \end{aligned}
\end{equation*}
\noindent Then, the projective completion of $C$ in $\mathbb{P}^3$, which we denote as $C'$, is given by 
\begin{equation*}
    \begin{aligned}
        & x^2+y^2-5t^2=0 \\
        & t^3+z^2t-xyz=0.
    \end{aligned}
\end{equation*}
It is easy to check $C'$ has three points at infinity: $(0:0:1:0),(1:i:0:0)$ and $(1:-i:0:0)$. The point $(0:0:1:0)$ is singular and the other two are smooth. In order to find a smooth projective model for $C$, we need to blow up $C'$ at the unique singular point $(0:0:1:0)$. Since the point $(0:0:1:0)$ is contained in the affine chart of $\mathbb{P}^3$ given by $z=1$, we consider $C'$ in this affine chart by setting $z=1$. Then $C'$ is defined by 
\begin{equation*}
    \begin{aligned}
        & x^2+y^2-5t^2=0 \\
        & t^3+t-xy=0.
    \end{aligned}
\end{equation*}
\noindent The point $(0:0:1:0)$ will correspond to the origin $O=(0,0,0)$ in this affine chart. So it suffices to find the blow-up of $C'$ at $O$. Let $X$ be the blow-up of $A^3$ at $O$, where $A^3$ stands for the standard affine 3-space. Then by definition, $X\subset A^3 \times \mathbb{P}^2$, and if we use $(x,y,t)$ and $(u:v:w)$ as the coordinates of $A^3$ and $\mathbb{P}^2$ respectively, then the total inverse image of $C'$ in $X$ is given by 
\begin{equation*}
    \begin{aligned}
        & x^2+y^2-5t^2=0 \\
        &t^3+t-xy=0 \\
        & xv=yu \\
        & xw=tu \\
        & yw=tv. \\
    \end{aligned}
\end{equation*}

\noindent Consider the affine chart of $A^3 \times \mathbb{P}^2$ given by $u=1$. In this chart, the defining equations above are simplified as: 
\begin{equation*}
    \begin{aligned}
        & x^2+y^2-5t^2=0 \\
        & t^3+t-xy=0 \\
        & xv=y \\
        & xw=t \\
        & yw=tv.
    \end{aligned}
\end{equation*}

\noindent Using the fact that $y=xv,t=xw$, we can simplify the equations above to get
\begin{equation*}
    \begin{aligned}
        & x^2+x^2v^2-5x^2w^2=0 \\
        & x^3w^3+xw-x^2v=0 \\
        & xv=y \\
        & xw=t.
    \end{aligned}
\end{equation*}

\noindent If $x=0$, then $y=t=0$ and $w,v$ can be any complex numbers. This will give us an affine plane $A^2$ which corresponds to the exceptional surface of the blow-up. Hence, we assume $x \neq 0$. Then we get the curve $\tilde{C}$ defined by 
\begin{equation*}
    \begin{aligned}
        & 1+v^2-5w^2=0 \\
        & x^2w^3+w-xv=0 \\
        & xv=y \\
        & xw=t.
    \end{aligned}
\end{equation*}

\noindent This is precisely the blow-up of $C'$ we are looking for. When $x=y=t=0$, we get $w=0$, $v=\pm i$. So after blow-up, $O$ of $C'$ is replaced by two points, and it's easy to check these two points are smooth, hence the singularity of $C'$ is resolved. Notice that in the process above, if we choose different affine charts($v=1$ or $w=1$), we will get the same result. So in summary, the point $(0:0:1:0)$ is a double point for $C'$ and we resolve the singularity by doing a single blow-up. Then we get a smooth projective model $\tilde{C}$ for our original curve $C$, which has 4 points at infinity.  

Now, let $Y$ be the plane curve defined by $x^2+y^2-5=0$. Let $\phi: C \rightarrow Y$ be the projection map onto the xy-plane, i.e. $\phi(x,y,z)=(x,y)$. Notice that the projective completion of $Y$ in $\mathbb{P}^2$, which we denote as $\tilde{Y}$, is defined by $x^2+y^2-5t^2=0$. Then $\tilde{Y}$ has two smooth points at infinity: $(1:i:0)$ and $(1:-i:0)$. It's easy to check that $\tilde{Y}$ is isomorphic to $\mathbb{P}^1$ and hence $Y$ is birational to $\mathbb{P}^1$. 

Since $\tilde{C}$ and $\tilde{Y}$ are the smooth projective models of $C$ and $Y$ respectively, we can extend $\phi:C \rightarrow Y$ to a morphism $\tilde{\phi}:\tilde{C} \rightarrow \tilde{Y}$. The points at infinity on $\tilde{C}$ will then be mapped to points at infinity on $\tilde{Y}$. For any point $(x,y)$ on $Y$, the preimage $\phi^{-1}(x,y)$ will contain 2 points except when $xy=\pm2$. Hence, the morphism $\tilde{\phi}$ is generically 2-1 except at 8 affine ramification points: $(2,\pm 1),(-2,\pm 1),(1,\pm 2),(-1,\pm 2)$. But we know $\tilde{Y}$ is isomorphic to $\mathbb{P}^1$, hence there is a morphism from $\tilde{C}$ to $\mathbb{P}^1$ with 8 ramification points and degree 2. So $\tilde{C}$ is a hyperelliptic curve with genus 3, using the Hurtwitz formula. 
\end{proof}

We now derive the Weierstrass normal form for $\tilde{C}$. To achieve this goal, we need to the precise map from $\tilde{C}$ to $\mathbb{P}^1$, which means we need the precise expression for the isomorphism $\tilde{\psi}: \tilde{Y} \rightarrow \mathbb{P}^1$. Let $\psi: Y \rightarrow \mathbb{C}$ be the stereographic projection from the point $(0,\sqrt{5})$. Then $\psi$ is given by

\begin{equation*}
    \begin{aligned}
        & (x,y) \mapsto \frac{\sqrt{5}x}{\sqrt{5}-y}, ~y\neq \sqrt{5} \\
        & (0,\sqrt{5}) \mapsto \infty.
    \end{aligned}
\end{equation*}

\noindent $\psi$ can be extended to $\tilde{\psi}:\tilde{Y} \rightarrow \mathbb{P}^1$ with $(1:\pm i:0) \mapsto \pm \sqrt{5}i$, which is the isomorphism we want, and $\tilde{\psi} \circ \tilde{\phi}:\tilde{C} \rightarrow \mathbb{P}^1$ is the map we are seeking. The 8 ramification points on $\mathbb{P}^1$ are then all affine and are given by
\begin{equation*}
    \frac{2\sqrt{5}}{\sqrt{5} \pm 1}, \frac{-2\sqrt{5}}{\sqrt{5} \pm 1}, \frac{\sqrt{5}}{\sqrt{5} \pm 2}, \frac{-\sqrt{5}}{\sqrt{5} \pm 2}.
\end{equation*} 

\noindent The Weierstrass normal form of $\tilde{C}$ is given by $y^2=f(x)$ where

\begin{equation*}
    f=\Pi_{i=1}^{8}(x-x_i)
\end{equation*}

\noindent such that the $x_i$'s are the 8 ramification points listed above. We can factor out the product and get the precise expression for the normal form as

\begin{equation*}
    y^2=x^8-105x^6+1400x^4-2625x^2+625.
\end{equation*}

We now study the tame symbol along $C$. Through similar computations as before, we get that the inverse of the tame symbol is again $\alpha=x^2-4$. We want to decide if the rational function $\alpha$ is a square in $k(C)$. Since $\tilde{C}$ is a smooth projective model for $C$, $k(C)\cong k(\tilde{C})$. So it suffices to view $\alpha$ as an element in $k(\tilde{C})$ and decide if it is a square there. As in the previous example, we need to find $div(\alpha)$ to achieve this goal. 

It is not easy to compute $div(\alpha)$ directly, so we apply the following trick. Let $\alpha^{'}=x^2-4$ be a rational function on $Y$, which is the projection of $C$ down to the xy-plane as defined previously. Again, we can view $\alpha^{'}$ as an element in $k(\tilde{Y})$, and it's easy to see $\alpha=\alpha^{'} \circ \tilde{\phi}$ since $\alpha$ and $\alpha^{'} \circ \tilde{\phi}$ agree on the affine part $C$ of $\tilde{C}$. Hence $div(\alpha)=div(\alpha^{'} \circ \tilde{\phi})=\tilde{\phi}^*(div(\alpha^{'}))$, where $\tilde{\phi}^{*}$ is the map on the divisor groups induced by $\tilde{\phi}$ such that 
\begin{equation}\label{induced map on divisors}
    \tilde{\phi}^{*}(Q)= \Sigma_{P \in \tilde{\phi}^{-1}(Q)}~e_{\phi}(P)(P)
\end{equation}

\noindent for any $Q \in \tilde{Y}$, and $e_{\phi}(P)$ is the ramification index at $P$. We refer readers to \cite{Silverman} for more details on $\tilde{\phi}^{*}$. 

Let's first compute $div(\alpha^{'})=div(x+2)+div(x-2)$. For $x+2$, it's easy to check it has two zeros at $(-2,1),(-2,-1)$ and two poles at infinity, all with multiplicities 1. So $div(x+2)=P_1+P_2-Q_1-Q_2$ where $P_1=(-2,1),P_2=(-2,-1)$ and $Q_1,Q_2$ are the two points at infinity. Similarly, $div(x-2)=P_3+P_4-Q_1-Q_2$ where $P_3=(2,1),P_4=(2,-1)$. Hence, $div(\alpha^{'})=P_1+P_2+P_3+P_4-2(Q_1+Q_2)$. Notice that $P_1,P_2,P_3,P_4$ are all ramification points of $\tilde{\phi}$ while $Q_1,Q_2$ are not. So 
\begin{equation*}
    div(\alpha)=\tilde{\phi}^*(div(\alpha^{'}))=2\tilde{P_1}+2\tilde{P_2}+2\tilde{P_3}+2\tilde{P_4}-2(\tilde{Q_1}+\tilde{Q_2}+\tilde{Q_3}+\tilde{Q_4})
\end{equation*}
\noindent where $\tilde{P_i}$ lies over $P_i$, $\tilde{Q_1},\tilde{Q_2}$ over $Q_1$ and $\tilde{Q_3},\tilde{Q_4}$ over $Q_2$. 

We have thus computed $div(\alpha)$ on $\tilde{C}$. We know $\tilde{C}$ is hyperelliptic and we have found its Weierstrass normal form. To make our computations later easier, we now transform all the divisors to divisors on the normal form. Under the stereographic projection we discussed previously, the two points of $\tilde{Y}$ at infinity are mapped to $\pm \sqrt{5}i$. The points $\tilde{P_i}$'s are all ramification points and lie over $\frac{2\sqrt{5}}{\sqrt{5} \pm 1}, \frac{-2\sqrt{5}}{\sqrt{5} \pm 1}$. So in the normal form, 
\begin{equation}\label{coordinates for tilde_Pi}
    \tilde{P_1}=(\frac{2\sqrt{5}}{\sqrt{5}+1},0), \tilde{P_2}=(\frac{2\sqrt{5}}{\sqrt{5}-1},0),\tilde{P_3}=(\frac{-2\sqrt{5}}{\sqrt{5}+1},0),\tilde{P_4}=(\frac{-2\sqrt{5}}{\sqrt{5}-1},0). 
\end{equation}

\noindent The x-coordinate for $\tilde{Q_1},\tilde{Q_2}$ is $\sqrt{5}i$ and the x-coordinate for $\tilde{Q_3},\tilde{Q_4}$ is $-\sqrt{5}i$. From now on, we assume we are working in the Weierstrass normal form and the points in $div(\alpha)$ has coordinates given as above. 

As mentioned above, we want to decide if $\alpha=g^2$ for some $g \in k(\tilde{C})$. If so, $div(\alpha)=2*div(g)$, and we must have $div(g)=\tilde{P_1}+\tilde{P_2}+\tilde{P_3}+\tilde{P_4}-(\tilde{Q_1}+\tilde{Q_2}+\tilde{Q_3}+\tilde{Q_4})$. Hence, it suffices to decide whether the divisor 
\begin{equation}\label{def of D}
    D=\tilde{P_1}+\tilde{P_2}+\tilde{P_3}+\tilde{P_4}-(\tilde{Q_1}+\tilde{Q_2}+\tilde{Q_3}+\tilde{Q_4})
\end{equation}

\noindent is principal or not. Notice that $deg(D)=0$, so $[ D]$ defines an element in $Pic^{0}(\tilde{C})$, the Jacobian of the curve $\tilde{C}$. We want to decide if the class $[D]$ is trivial. To check the triviality of $[D]$, we apply the Mumford representation and Cantor's algorithm \cite{Cantor}, \cite{Sutherland}, which we now describe. 

Let $X$ be a hyperelliptic curve. Fix a Weierstrass normal form $y^2=f(x)$ for $X$. We will be primarily working with real hyperelliptic curves. Hence, for the remaining discussion, we assume our hyperelliptic curve is real with two smooth points $P_{\infty},\bar{P_{\infty}}$ at infinity, and $f$ has even degree.  

\begin{definition}
    An effective divisor $D=\Sigma P_i$ on $X$ is semi-reduced if $P_i \neq \bar{P_j}$ for any $i \neq j$. A semi-reduced divisor whose degree is less than or equal to $g(X)$ is said to be reduced. A semi-reduced affine divisor is a semi-reduced divisor with $P_i \neq P_{\infty} ~or ~\bar{P_{\infty}}$ for all $i$. 
\end{definition}  

A semi-reduced affine divisor $D=\Sigma P_i$ can be described by its Mumford representation $div[u,v]$, which we now define. Let $P_i=(x_i,y_i)$. Define $u(x)=\Pi_i (x-x_i)$ and let $v$ be the unique polynomial of degree less than $deg(u)$ such that $u|f-v^2$ and $v(x_i)=y_i$. Conversely, if $u,v \in K[x]$ satisfy: $u$ is monic, $deg(v)<deg(u)$ and $u|f-v^2$, then we say $u,v$ satisfy the Mumford criteria. In this case, write $u(x)=\Pi_i (x-x_i)$. Define $P_i:=(x_i,v(x_i))$ and $div[u,v]:=\Sigma P_i$. Then it can be shown that $div[u,v]$ is semi-reduced. The relationship between semi-reduced divisors and the Mumford representations can be summarized into the following theorem. 

\begin{proposition}\label{1-1 corr betwen SRAD and MR}
(\cite{Sutherland}) There is a 1-1 correspondence between semi-reduced affine divisors on a hyperelliptic curve $X$ and Mumford representations $div[u,v]$.
\end{proposition}

We can now describe every semi-reduced affine divisor by its Mumford representation, but we want to generalize this representation to any element in $Pic^{0}(X)$ so that we can do computations on those representations. To achieve this goal, we need the following theorem. 

\begin{proposition}\label{unique rep of D}
(\cite{Sutherland}, Proposition 2.4.) Let $X$ be a hyperelliptic curve of genus $g$ and let $D_{\infty}=\lceil \frac{g}{2} \rceil P_{\infty}+\lfloor \frac{g}{2} \rfloor \bar{P_{\infty}}$, then each divisor class $[D] \in Pic^{0}(X)$ can be uniquely written as $[D_0-D_{\infty}]$ where $D_0$ is an effective divisor of degree $g$ whose affine part is reduced.
\end{proposition}
    
\noindent Based on Proposition \ref{1-1 corr betwen SRAD and MR} and \ref{unique rep of D}, we can easily deduce the following fact: every $[D] \in Pic^{0}(X)$ can be uniquely represented by a triple $(u,v,n)$ where $(u,v)$ satisfy the Mumford criteria and $div[u,v]$ is reduced. The triple $(u,v,n)$ will then correspond to the divisor

\begin{equation}\label{def of div_uvn}
    div[u,v,n]:=div[u,v]+nP_{\infty}+(g-deg(u)-n)\bar{P_{\infty}}-D_{\infty}.
\end{equation}

\begin{definition}
    \eqref{def of div_uvn} is called the Mumford representation for $[D] \in Pic^{0}(X)$. In particular, the trivial class $[0]$ has Mumford representation given by $div[1,0,\lceil \frac{g}{2} \rceil]$.
\end{definition} 

\noindent Now, we get a more compact description of elements in $Pic^{0}(X)$ using the Mumford representation. We are eventually interested in adding two elements in the Jacobian, i.e. given $[D_1]=div[u_1,v_1,n_1], ~[D_2]=div[u_2,v_2,n_2]$, how to get $(u,v,n)$ such that $div[u,v,n]=[D_1]+[D_2]$? We introduce Cantor's algorithm to answer this question. 

We first introduce an intermediate notation that will correspond to divisors whose affine part is semi-reduced instead of reduced. Given a semi-reduced affine divisor $div[u,v]$ with $deg(u) \leq 2g$ and an integer $n$ with $0 \leq n \leq 2g-deg(u)$, we define

\begin{equation}\label{def of div_uvn_star}
    div[u,v,n]^{*}:=div[u,v]+nP_{\infty}+(2g-deg(u)-n)\bar{P_{\infty}}-2D_{\infty}
\end{equation}

\noindent With this notation, the Cantor's algorithm can be summarized into the following 3 algorithms (We refer the reader to \cite{Sutherland} and \cite{Cantor} for details).

\bigskip
\noindent \textbf{Algorithm} PRECOMPUTE

\noindent Given $f(x) = x^{2g+2}+f_{2g+1}x^{2g+1}+\cdots +f_1x+f_0$, compute the unique monic $V(x)$ for which $deg(f-V^2) \leq g$. 

\noindent 1. Set $V_{g+1}:=1$

\noindent 2. For $i=g,g-1,...0$, compute $c:=f_{g+1+i} - \sum_{j=i+1}^{g+1} V_jV_{g+1+i-j}$ and set $V_i := c/2$. 

\noindent 3. Output $V(x):=x^{g+1}+V_gx^g+\cdots V_1x+V_0$. 

\bigskip
\noindent \textbf{Algorithm} COMPOSE

\noindent Given $div[u_1,v_1,n_1]$ and $div[u_2,v_2,n_2]$, compute $div[u_3,v_3,n_3]^{*}$ such that:

\begin{center}
$div[u_1,v_1,n_1]+div[u_2,v_2,n_2] ~\thicksim~ div[u_3,v_3,n_3]^{*}$

\end{center}

\noindent 1. Use the Euclidean algorithm to compute monic $w:=gcd(u_1,u_2,v_1+v_2) \in K[x]$ and $c_1,c_2,c_3 \in K[x]$ such that $w=c_1u_1 +c_2u_2+c_3(v_1+v_2)$

\noindent 2. Let $u_3:= u_1u_2/w^2$ and let $v_3:=(c_1u_1v_2+c_2u_2v_1+c_3(v_1v_2+f))/w ~~mod~ u_3$, where $f ~mod~ g$ means the unique polynomial $h$ such that $deg(h)<deg(g)$ and $g|f-h$. 

\noindent 3. Output $div[u_3,v_3,n_1+n_2+deg(w)]^{*}$. 

\bigskip
\noindent \textbf{Algorithm} ADJUST

\noindent Given $div[u_1,v_1,n_1]^{*}$ with $deg(u_1) \leq g+1$, compute $div[u_2,v_2,n_2]$ such that:

\begin{center}
$div[u_1,v_1,n_1]^{*} ~\thicksim~ div[u_2,v_2,n_2]$

\end{center}

\noindent 1. If $n_1 \geq \lceil g/2 \rceil$ and $n_1 \leq \lceil 3g/2 \rceil - deg(u_1)$, then output $div[u_1,v_1,n_1- \lceil g/2 \rceil]$ and terminate.

\noindent 2. If $n_1 < \lceil g/2 \rceil$, let $\hat{v_1}:=v_1 - V+(V ~mod~ u_1)$, let $u_2$ be $(f-\hat{v_1}^2)/u_1$ made monic, let $v_2:=-\hat{v_1} ~mod~ u_2$, and let $n_2:=n_1+g+1-deg(u_2)$. 

\noindent 3. If $n_1 \geq \lceil g/2 \rceil$, let $\hat{v_1}:=v_1 + V-(V ~mod~ u_1)$, let $u_2$ be $(f-\hat{v_1}^2)/u_1$ made monic, let $v_2:=-\hat{v_1} ~mod~ u_2$, and let $n_2:=n_1+deg(u_1)-(g+1)$. 

\noindent 4. Output ADJUST($div[u_2,v_2,n_2]^{*}$). 

\bigskip

We give a quick explanation of the algorithms. PRECOMPUTE is an auxiliary step whose output will be used in the 3rd step ADJUST. The 2nd step COMPOSE is the key step that adds the Mumford representation, but the output $div[u,v]$ may not be reduced (the Mumford representation for elements in the Jacobian requires $div[u,v]$ to be reduced), ie: $deg(u)$ may be greater than $g$. If that's the case, then we need the 3rd step ADJUST to reduce the degree of $u$ and get the correct $(u,v,n)$ such that $div[u,v]$ is reduced. Notice that, strictly speaking, ADJUST only deals with the situation where $deg(u) \leq g+1$. We need another algorithm to deal with cases when $deg(u)>g+1$. But this does not happen in the examples we will consider(as we shall see later). Therefore, ADJUST will be enough for our purpose. 

With all these concepts, we go back to our example. In our case, $g(\tilde{C})=3,D_{\infty}=2P_{\infty}-\bar{P_{\infty}}$ and $[0]=div[1,0,2]$. We want to decide if $D=\tilde{P_1}+\tilde{P_2}+\tilde{P_3}+\tilde{P_4}-(\tilde{Q_1}+\tilde{Q_2}+\tilde{Q_3}+\tilde{Q_4})$ is principal, i.e. if $[D]$ is trivial. Consider 
\begin{equation*}
    \begin{aligned}
        & D_1=\tilde{P_1}+\tilde{P_2}-(\tilde{Q_1}+\tilde{Q_2}) \\
        & D_2=\tilde{P_3}+\tilde{P_4}-(\tilde{Q_3}+\tilde{Q_4}).
    \end{aligned}
\end{equation*}
\noindent Then $D=D_1+D_2$ and $[D]=[D_1]+[D_2]$. For $D_1$, we know $\tilde{Q_1}, \tilde{Q_2}$ lie over $\sqrt{5}i$. Consider $h \in K(\tilde{C})$ given by $(x,y) \mapsto x-\sqrt{5}i$. Then $div(h)=\tilde{Q_1}+\tilde{Q_2}-P_{\infty}-\bar{P_{\infty}}$. So 
\begin{equation*}
    D_1 \sim \tilde{P_1}+\tilde{P_2}-P_{\infty}-\bar{P_{\infty}}=\tilde{P_1}+\tilde{P_2}+P_{\infty}-2P_{\infty}-\bar{P_{\infty}}=\tilde{P_1}+\tilde{P_2}+P_{\infty}-D_{\infty}=div[u,v,1]
\end{equation*}
\noindent where $u=(x-x_1)(x-x_2)$ such that $x_1,x_2$ are the x-coordinates for $\tilde{P_1},\tilde{P_2}$ respectively, i.e. $x_1=\frac{2\sqrt{5}}{\sqrt{5} + 1},x_2=\frac{2\sqrt{5}}{\sqrt{5} - 1}$. So $u=x^2-5x+5$ and $v=0$. Then $D_1 \sim div[x^2-5x+5,0,1]$. Similarly, $D_2 \sim div[x^2+5x+5,0,1]$. 

Now, we have the Mumford representations for $D_1$ and $D_2$, and we want to find the Mumford representation for $D=D_1+D_2$. To find the sum of two Mumford representations, we apply the Cantor's algorithm. First, we compute the auxiliary polynomial $V(x)$, which is given by $V(x)=x^4-52.5x^2-678.125$. It can be easily checked that this is indeed a monic polynomial with $deg(f-V^2) \leq g$, where $f$ is the defining polynomial for the hyperelliptic curve. Then we use the COMPOSE algorithm to compute $(u_3,v_3,n_3)$ such that
\begin{equation*}
    div[u_3,v_3,n_3]^{*} \sim div[x^2-5x+5,0,1]+div[x^2+5x+5,0,1].
\end{equation*}
\noindent We now follow the steps outlined in the COMPOSE algorithm. Clearly $u_1, u_2$ are relatively prime and hence $w=1, c_3=0$. Then $u_3=u_1u_2=x^4-15x^2+25, v_3=0$ and $n_3=2$. In our case, $g=3$ and hence $deg(x^4-15x^2+25) \leq g+1$. So we follow the ADJUST algorithm to find $(u_2,v_2,n_2)$ such that
\begin{equation*}
    div[u_2,v_2,n_2] \sim div[u_1,v_1,n_1]^{*} = div[x^4-15x^2+25,0,2]^{*}.
\end{equation*}
\noindent In our case, $n_1=2 \geq \lceil g/2 \rceil$, but $n_1> \lceil 3g/2\rceil -deg(u_1)$. So we go to step 3 of the algorithm. $V ~mod~ u_1 =V-u_1$, hence $\hat{v_{1}}=u_1=x^4-15x^2+25$. Then $u_2=x^2$, $v_2=-25$ and $n_2=2$. Now, $n_2 \geq \lceil g/2 \rceil$ and $n_2 \leq \lceil 3g/2 \rceil -deg(u_2)$. So we go to step 1 of the ADJUST algorithm and output $div[x^2,-25,0]$ as our final result. 

Therefore, we have $D \sim div[x^2,-25,0]$. Clearly, this is not the Mumford representation for $[0]$, so $D$ is not principal, and hence $\alpha=x^2-4$ is not a square in $k(\tilde{C})$. So once again, $A_{k(S_3)}$ doesn't extend over $S_3$.

\bibliographystyle{spmpsci}
\bibliography{Ref.bib}

@book{Thurston,
	Author = {William P. Thurston},
	Publisher = {American Mathematical Society},
	Title = {The Geometry and Topology of Three-Manifolds},
	Year = {2022}}

@article{Culler_Shalen_1,
    author = {M. Culler and P. B. Shalen},
    title = {Varieties of group representations and splittings of 3-manifolds},
    journal = {Ann. of Math.},
    year = {1983},
    volume = {117},
    pages = {109-146}
}

@article{Chin_Reid_Stover,
    author = {T. Chinburg and A. W. Reid and M. Stover},
    title = {Azumaya algebras and canonical components},
    journal = {International Mathematics Research Notices},
    year = {2020},
    volume = {2022},
    pages = {4969-5036}
}

@article{Rouse,
    author = {Nicholas Rouse},
    title = {Arithmetic of the canonical component of the knot $7_4$},
    journal = {New York J. Math.},
    year = {2021},
    volume = {27},
    pages = {1494-1523}
}

@article{Miller,
    author = {Nicholas Miller},
    title = {Azumaya algebras and once-punctured torus bundles},
    journal = {arXiv:2303.16309}
}

@book{Mac_and_Reid,
	Author = {C. Maclachlan and A. W. Reid},
	Publisher = {Graduate Texts in Mathematics. Springer-Verlag},
	Title = {The Arithmetic of Hyperbolic 3-Manifolds},
	Year = {2003}}

@book{Hartshorne,
	Author = {R. Hartshorne},
	Publisher = {Graduate Texts in Mathematics. Springer-Verlag},
	Title = {Algebraic Geometry},
	Year = {1977}}

@book{Silverman,
	Author = {J. H. Silverman},
	Publisher = {Graduate Texts in Mathematics. Springer},
	Title = {The Arithmetic of Elliptic Curves},
	Year = {2009}}

@book{Shafarevich,
	Author = {Igor R. Shafarevich},
	Publisher = {Springer-Verlag},
	Title = {Basic Algebraic Geometry 1.},
	Year = {1988}}

@article{Sutherland,
    author = {Andrew Sutherland},
    title = {Fast Jacobian arithmetic for hyperelliptic curves of genus 3},
    journal = {The Open Book Series},
    year = {2019},
    volume = {2},
    pages = {425-442}
}

@article{Cantor,
    author = {David Cantor},
    title = {Computing in the Jacobian of a hyperelliptic curve},
    journal = {Mathematics of Computation},
    year = {1987},
    volume = {48},
    pages = {95-101}
}

@book{Thurston_2,
	Author = {William P. Thurston},
	Publisher = {Princeton University Press},
	Title = {Three-Dimensional Geometry and Topology},
	Year = {1997}}

@book{Ratcliffe,
	Author = {John Ratcliffe},
	Publisher = {Graduate Texts in Mathematics. Springer},
	Title = {Foundations of Hyperbolic Manifolds},
	Year = {2007}}

@article{Landes,
    author = {Emily Landes},
    title = {Identifying the canonical component for the Whitehead link},
    journal = {Math. Res. Lett.},
    year = {2011},
    volume = {18},
    pages = {715-731}
}

@article{Harada,
    author = {Shinya Harada},
    title = {Canonical components of character varieties of arithmetic two bridge link complements},
    journal = {Eur. J. Math.},
    year = {2023},
    volume = {9},
}

@book{Eisenbud-Harris,
	Author = {David Eisenbud and Joe Harris},
	Publisher = {Cambridge University Press},
	Title = {3264 $\&$ All That - Intersection Theory in Algebraic Geometry},
	Year = {2016}}

@article{Petersen,
    author = {Kathleen Petersen and Anh. T. Tran},
    title = {Identifying the canonical component for the Whitehead link},
    journal = {Algebraic $\&$ Geometric Topology},
    year = {2015},
    volume = {15},
    pages = {3569-3598}
}

@book{Gille,
	Author = {Philippe Gille and Tamas Szamuely},
	Publisher = {Cambridge University Press},
	Title = {Central Simple Algebras and Galois Cohomology},
	Year = {2006}}

@book{Poonen,
	Author = {Bjorn Poonen},
	Publisher = {Graduate Studies in Mathematics},
	Title = {Ratioinal Points on Varieties},
	Year = {2017}}

@book{Rolfsen,
	Author = {Dale Rolfsen},
	Publisher = {Mathematics Lecture Series},
	Title = {Knots and Links},
	Year = {1976}}

@book{Zieschang,
	Author = {Gerhard Burde and Heiner Zieschang},
	Publisher = {Walter de Gruyter},
	Title = {Knots},
	Year = {2003}}

@article{Gehring,
    author = {F. W. Gehring and C. Maclachlan and G. J. Martin},
    title = {Two-generator arithmetic Kleinian groups. II},
    journal = {Bull. London Math. Soc.},
    year = {1998},
    volume = {30},
    pages = {258-266}
}

@article{HT1985,
    author = {A. Hatcher and W. Thurston},
    title = {Incompressible surfaces in 2-bridge knot complements},
    journal = {Inventiones math.},
    year = {1985},
    volume = {79},
    pages = {225-246}
}

@article{PT2015,
    author = {Kathleen Petersen and Anh. T. Tran},
    title = {Character varieties of double twist links},
    journal = {Algebr. Geom. Topol.},
    year = {2015},
    volume = {15},
    pages = {3569-3598}
}

@article{CGLS1987,
    author = { Marc Culler and C. McA. Gordon and J. Luecke and Peter B. Shalen},
    title = {Dehn surgery on knots},
    journal = {Annals of Mathematics},
    year = {1987},
    volume = {125},
    pages = {237-300}
}

@article{MPV2011,
    author = {Melissa L. Macasieb and Kathleen L. Petersen and Ronald M. van Luijk},
    title = {On character varieties of two-bridge knot groups},
    journal = {Proceedings of the London Mathematical Society},
    year = {2011},
    volume = {103},
    pages = {473-507}
}

@article{BK2013,
    author = {K. Baker and Kathleen L. Petersen},
    title = {Character varieties of once-punctured torus bundles with tunnel number one},
    journal = {International Journal of Mathematics},
    year = {2013},
    volume = {24}
}

@article{RR2023,
    author = {Alan W. Reid and Nicholas Rouse},
    title = {Infinitely many knots with non-integral trace},
    journal = {Experimental Mathematics},
    year = {2023},
    volume = {32},
    pages = {514-526}
}

@article{LR2003,
    author = {D.D. Long and A.W. Reid},
    title = {Integral points on character varieties},
    journal = {Math. Ann.},
    year = {2003},
    volume = {325},
    pages = {299–321}
}

@article{FH1988,
    author = {W. Floyd and A. Hatcher},
    title = {The Space of Incompressible Surfaces in a 2-Bridge Link Complement},
    journal = {Transactions of the American Mathematical Society},
    year = {1988},
    volume = {305},
    pages = {575-599}
}

@article{HS2007,
    author = {Jim Hoste and Patrick D. Shanahan},
    title = {Computing boundary slopes of 2-bridge links},
    journal = {Mathematics of Computation},
    year = {2007},
    volume = {76},
    pages = {1521-1545}
}

@book{Milne1980,
	Author = {J. S. Milne},
	Publisher = {Princeton University Press},
	Title = {\'Etale cohomology},
        Volume ={33 of Princeton Mathematical Series},
	Year = {1980}}

@article{COP2002,
    author = {J.-L. Colliot-Th\'el\`ene and M. Ojanguren and R. Parimala},
    title = {Quadratic forms over fraction elds of two-dimensional Henselian rings and Brauer groups of related schemes},
    journal = {Algebra, arithmetic and geometry, Part I, II (Mumbai, 2000)},
    year = {2002},
    volume = {16},
    pages = {185-217}
}

\nocite{HT1985, CGLS1987, RR2023, FH1988, HS2007}

\end{document}